\documentclass[11pt,a4paper]{amsart}

\usepackage[T1]{fontenc}

\usepackage{mathtools}
\usepackage{amssymb}
\usepackage{amsthm}
\usepackage{amsaddr}
\usepackage{mathrsfs}
\usepackage{tikz-cd}
\usepackage{enumitem}
\usepackage{indentfirst}
\usepackage{stackrel}
\usepackage{csquotes}
\usepackage{multicol}

\usepackage[geometry]{ifsym}
\usepackage{geometry}

\usepackage[english]{babel}

\usetikzlibrary{babel}

\newcommand{\id}{\textsf{id}}

\newcommand{\comma}{\ensuremath{\downarrow}}
\newcommand{\cat}[1]{\ensuremath{\mathcal{#1}}}

\newcommand{\intcat}[1]{\mathscr{#1}}
\newcommand{\catname}[1]{\mathsf{#1}}
\newcommand{\Set}{\catname{Set}}
\newcommand{\Cat}{\catname{Cat}}
\newcommand{\Span}{\catname{Span}}
\newcommand{\Fam}{\catname{Fam}}
\newcommand{\Desc}{\catname{Desc}}
\newcommand{\Top}{\catname{Top}}
\newcommand{\Mod}{\catname{Mod}}
\newcommand{\Alg}{\catname{Alg}}

\newcommand{\iso}{\cong}
\newcommand{\adj}{\dashv}
\newcommand{\dash}{\text{-}}

\DeclareMathOperator{\PsEq}{\mathsf{PsEq}}
\DeclareMathOperator{\ob}{\catname{ob}}
\DeclareMathOperator{\word}{\catname{W}}

\newtheorem{lemma}{Lemma}[section]
\newtheorem{proposition}[lemma]{Proposition}
\newtheorem{corollary}[lemma]{Corollary}
\newtheorem{theorem}[lemma]{Theorem}

\theoremstyle{definition}

\newtheorem{remark}[lemma]{Remark}

\title[Descent for multicategory functors]
{Descent for internal multicategory functors}

\date{April 26, 2021}

\author[Rui Prezado]{Rui Prezado} %1
\address[1,2]{University of Coimbra, CMUC, Department of Mathematics,
Portugal}
\email[1]{rui.prezado@student.uc.pt}

\author[Fernando Lucatelli Nunes]{Fernando Lucatelli Nunes} %2
\address[2]{Utrecht University, The Netherlands}
\email[2]{f.lucatellinunes@uu.nl}

\thanks{The first author was supported by the grant PD/BD/150461/2019 funded
by Fundação para a Ciência e Tecnologia (FCT). Both authors were supported by
the Centre for Mathematics of the University of Coimbra - UIDB/00324/2020,
funded by the Portuguese Government through FCT/MCTES.  This work was also
supported through the programme ``Oberwolfach Leibniz Fellows'' by the
Mathematisches Forschungsinstitut Oberwolfach in 2022.}



\begin{document}

  \begin{abstract}
    We give sufficient conditions for effective descent in categories of
(generalized) internal multicategories. Two approaches to study effective
descent morphisms are pursued. The first one relies on establishing the
category of internal multicategories as an equalizer of categories of diagrams.
The second approach extends the techniques developed by Ivan Le Creurer in his
study of descent for internal essentially algebraic structures.

  \end{abstract}

  \maketitle

  \tableofcontents

  \section*{Introduction}
    \label{Introduction} 
    Let \( \cat B \) be a category and \(p \colon x \to y \) a morphism in \( \cat
B \) such that pullbacks along \(p\) exist. We say that \( p \) is an
\textit{effective descent \emph(descent\emph)} morphism whenever the
change-of-base functor
\begin{equation*}
  p^\ast \colon \cat B \comma y \to \cat B \comma x
\end{equation*}
is monadic (premonadic). The main subject of this note, the study of
effective descent morphisms, is at the core of Grothendieck Descent Theory
(see \textit{e.g.} \cite{JT94, JST04}) and its applications (see, for
instance, \cite{BR01}). 

Except for the case of locally cartesian closed categories, the full
characterization of effective descent morphisms is far from trivial in
general.  The topological descent case is the main example of such a
challenging problem (see the characterization in \cite{RT94} and the
reformulation in \cite{CH02}).

The notion of \( (T,\cat V) \)-categories, introduced in \cite{CT03},
generalizes both enriched categories and various notions of \textit{spaces}.
By studying effective descent morphisms in categories of \( (T,\cat V)
\)-categories, Clementino and Hofmann were able to give further descent
results and understanding in various contexts, including, for instance, the
reinterpretation of the topological results mentioned above and many other
interesting connections (see, for instance, \cite{CH04, CJ11, CH12, CH17}).

On one hand, since they were mainly concerned with topological results, their
study focused on the case where \(\cat V\) is a quantale, and there is no
obvious way to generalize their approach to more general monoidal categories
\(\cat V\). On the other hand, their work, together with the characterization
of effective descent morphisms for the category of internal categories (see
\cite[Section~6]{JST04} and \cite{Cre99}), have raised interest in further
studying effective descent morphisms in categories of \textit{generalized}
categorical structures.

With this in mind, \cite[Lemma~9.10]{Luc18} showed that we can embed the
category of \(\cat V\)-enriched categories (with \(\cat V\) lextensive) in the
category of internal categories in \(\cat V\). From this embedding,
\cite[Theorem~1.6]{Luc18} provides sufficient conditions for effective descent
morphisms in \(\cat V\)-categories.  However, the literature still lacks
results for \((T,\cat V)\)-categories for a non-trivial \(T\) and an extensive
\(\cat V\).

The present note is part of a project which aims to study descent and
Janelidze-Galois theory within the realm of generalized multicategories and
other categorical structures. The first aim of this project consists of
studying effective descent morphisms in categories of generalized
multicategories.

While the definition of \( (T,\cat V) \)-categories generalizes that of
enriched categories, the definitions of internal \( T \)-multicategories in \(
\cat B\), for \( T \) a (cartesian) monad and \( \cat B \) with pullbacks,
introduced in \cite[p.~8]{Bur71} and \cite[Definition~4.2]{Her00}, generalize
the notion of internal categories.  Following this viewpoint and the approach
of \cite[Theorem~1.6]{Luc18}, in order to study effective descent morphisms
between more general \( (T,\cat V) \)-categories, the first step is to study
effective descent morphisms of categories of \textit{internal}
\(T\)-multicategories, \textit{which is the aim of the present paper.}

The main contributions of our present work consist of two approaches to the
problem of finding effective descent morphisms between internal
multicategories. We explain, below, the key ideas of our first approach, which
is the main subject of Section \ref{First approach}. 

As a special case of \cite[Theorem~9.2]{Luc18} (see Proposition
\ref{pseq.descent}), given a pseudo-equalizer (\textit{iso-inserter})
\begin{equation*}
  \begin{tikzcd}
    \PsEq(F,G) \ar[r,"I"]
    & \cat C \ar[r,shift left,"F"] \ar[r,shift right,"G",swap]
    & \cat D
  \end{tikzcd}
\end{equation*}
of categories with pullbacks and pullback preserving functors, \( p \)
is of effective descent whenever \( FIp \) is of descent and \( Ip \) is of
effective descent. Therefore, whenever (effective) descent morphisms in \(
\cat C \) and \( \cat D \) are well-understood, we find tractable, sufficient
conditions for effective descent in \( \PsEq(F,G) \).

We establish the category \( \Cat(T, \cat B) \) of internal
\(T\)-multicategories in \( \cat B \) as an equalizer consisting of a category
of models of a finite limit sketch and categories of diagrams (Lemma
\ref{multicat.as.eq}), which is fully embedded in the corresponding
pseudo-equalizer (Theorem \ref{embedding}). Since descent in categories of
models of a finite limit sketch were studied in \cite[Section 3.2]{Cre99}, and
categories of diagrams are well-understood, we obtain sufficient conditions
for effective descent in the pseudo-equalizer by the result mentioned above
(Lemma \ref{part.1}).

Finally, we find that the embedding of \( \Cat (T, \cat B) \) into the
pseudo-equalizer reflects effective descent morphisms (Lemma \ref{part.2}),
getting, then, our first result. Namely, a functor \( p \) of internal
\(T\)-multicategories is effective for descent whenever 
\begin{itemize}
  \renewcommand\labelitemi{--}
  \item
    \(Tp_1\) is an effective descent morphism in \( \cat B \),
  \item
    \(Tp_2\) is a descent morphism in \( \cat B \), 
  \item
    \(p_3\) is an almost descent morphism in \( \cat B \),
\end{itemize}
\textit{where \( p_i \) is the component of \(p\) between the objects of
\(i\)-tuples of composable morphisms} (Theorem \ref{first.theorem}).

Our second approach to the problem is presented in Section \ref{Second
approach}, which extends the work of \cite{Cre99} on effective descent
morphisms between internal structures. We observe that the same techniques
employed in Le Creurer's work can be applied to the ``sketch'' of internal
\(T\)-multicategories. With these techniques, we were able to \textit{refine}
our result on effective descent morphisms. We prove that functors \(p\) such
that
\begin{itemize}
  \renewcommand\labelitemi{--}
    \item
      \( p_1 \) is an effecive descent morphism in \( \cat B \),
    \item
      \( p_2 \) is a descent morphism in \( \cat B \),
    \item
      \( p_3 \) is an almost descent morphism in \( \cat B \),
\end{itemize}
are effective descent morphisms in \( \Cat(T,\cat B) \).

The techniques exploited in Section \ref{Second approach} proved to be more
suitable to our context of internal structures. However, the approach given 
there cannot be trivially applied to other generalized (enriched) categorical
structures. Thus, Section \ref{First approach} has expository value and its
techniques are especially relevant to our future work in descent theory of
generalized (enriched) categorical structures.

After fixing some notation on Section \ref{Preliminaries}, 
we recall some basic aspects on effective descent morphisms in Section
\ref{Effective descent morphisms}. Then, we study the equalizer that gives the
category of internal \( T \)-multicategories and its corresponding
pseudo-equalizer in Section \ref{Multicategories and pseudoequalizers}.
Afterwards, we discuss each approach to our main problem in the two subsequent
sections. We end the paper with a discussion of examples of cartesian monads
and internal multicategories.

  \section{Preliminaries}
    \label{Preliminaries} 
    Let \( J \colon \cat B \to \cat C \) be a diagram with a limit \( (\lim J,
\lambda) \). For any cone \( \gamma_b \colon x \to Jb \), there exists a
unique morphism \( f \colon x \to \lim J \) such that \( \gamma_b = \lambda_b
\circ f \) for all \(b\) in \( \cat B \). We denote \(f\) as \((\gamma_b)_{b
\in \ob \, \cat B}\).  As an example, let \( \cat B \) be a category with
pullbacks, and \( \intcat C \) an internal category. The object of pairs of
composable morphisms is given by the pullback:
\begin{equation*}
  \begin{tikzcd}
    \intcat C_2 \ar[r,"d_0"] \ar[d, swap, "d_2"] & \intcat C_1 \ar[d,"d_1"] \\
    \intcat C_1 \ar[r, swap, "d_0"] & \intcat C_0
  \end{tikzcd}
\end{equation*}
Thus, if we have morphisms \( g \colon X \to \intcat C_1 \) and \( f \colon X
\to \intcat C_1 \) with \( d_1 \circ g = d_0 \circ f \), we write \( (g,\, f)
\) for the uniquely determined morphism \( X \to \intcat C_2 \).  Furthermore,
we denote the internal composition by \( g \bullet f = d_1 \circ (g,\, f) \),
where \(d_1 \colon \intcat C_2 \to \intcat C_1 \) is the composition morphism.
Likewise, we can talk about tuples of composable morphisms, an idea we apply
to \(T\)-multicategories.

Another remark on notation: in a category \( \cat B \) with a choice of
pullbacks, we write
\begin{equation*}
  \begin{tikzcd}
    v \ar[r,"\epsilon_f"] \ar[d,"p^*f",swap] & w \ar[d,"f"] \\
    x \ar[r, swap,"p"] & y
  \end{tikzcd}
\end{equation*}
for the chosen pullback of \( f \) along \( p \). It is clear that the
\textit{change-of-base} \( p^* \colon \cat B \comma y \to \cat B \comma x \)
defines a functor right adjoint to \( p_!  \colon \cat B \comma x \to \cat B
\comma y \) with counit \( \epsilon \). For a morphism \( h \colon f \to g \)
in \( \cat B \comma y \) (that is, \(f = g \circ h\)), write \( p^*_h \) for
the unique morphism \( p^*f \to p^*g \) such that \( \epsilon_g \circ p^*_h =
h \circ \epsilon_f \).

  \section{Effective descent morphisms}
    \label{Effective descent morphisms}
    We recall some known facts about effective descent morphisms. In a category \(
\cat B \) with chosen pullbacks along \(p\), the category \( \Desc(p) \) of
descent data for a morphism \( p \colon x \to y \) in \( \cat B \) is defined
as the category of algebras for the monad \( p^*p_! \). Explicitly, objects
are pairs of morphisms \( (a\colon w \to x, \gamma \colon v \to w) \)
satisfying
\begin{itemize}
  \renewcommand\labelitemi{--}
  \item
    \( p^*(p\circ a) = a \circ \gamma \), that is, \( \gamma \) is a
    morphism \( p^*(p\circ a) \to a \) in \( \cat B \comma x \),
  \item
    \( \gamma \circ p^*_{\epsilon_{p \circ a}} = \gamma \circ p^*_\gamma \),
    the multiplication law (note that \( p \circ a \circ \gamma = p \circ
    p^*(p \circ a) \), so that we may apply \(p^*\)),
  \item
    \( \gamma \circ (a,\id) = \id \), the unit law, where \( (a,\id) \) is the
    unique morphism such that \( a = p^*(p\circ a) \circ (a,\id) \) and \( \id
    = \epsilon_{p \circ a} \circ (a,\id) \).
\end{itemize}
A morphism \( (a,\gamma) \to (b,\theta) \) of descent data is a morphism \( f
\) with \( a = b \circ f \) such that \( f \circ \gamma = \theta \circ p^*_f
\).

Further recall the \textit{Eilenberg-Moore factorization} of \( p^* \):
\begin{equation*}
  \begin{tikzcd}
    \cat B \comma y \ar[rr,"\mathcal K^p"] 
                    \ar[rd,"p^*",swap] 
      && \Desc(p)   \ar[ld,"\mathcal U^p"] \\
    & \cat B \comma x
  \end{tikzcd}
\end{equation*}
Here, \( \mathcal U^p \) is the forgetful functor, and \( \mathcal K^p \) is
commonly denoted the \textit{comparison functor}. We say a morphism \(p\) is
\begin{itemize}
  \renewcommand\labelitemi{--}
  \item
    an \textit{almost descent morphism} if \( \mathcal K^p \) is faithful,
  \item
    a \textit{descent morphism} if \( \mathcal K^p \) is fully faithful,
  \item
    an \textit{effective descent morphism} if \( \mathcal K^p \) is an
    equivalence.
\end{itemize}

By the Bénabou-Roubaud theorem (originally proven in \cite{BR70}, see, for
instance, \cite[p.~258]{JT94} or \cite[Theorem~7.4~and~Theorem~8.5]{Luc18} for
generalizations), this is equivalent to the classical formulation of the
descent category w.r.t. the basic (bi)fibration. 

As a consequece of Beck's monadicity theorem, we may characterize (almost)
descent morphisms (also check \cite[Corollary 0.3.4]{Cre99} and \cite[Theorem
3.4]{JST04}):

\begin{proposition}
  \label{uniregepi.is.descent} 
  In a category \( \cat B \) with finite limits, pullback-stable epimorphisms
  are exactly the almost descent morphisms, and pullback-stable regular
  epimorphisms in are exactly the descent morphisms.
\end{proposition}

\begin{proof}
  Let \(p \colon x \to y\) be a morphism in \( \cat B \). \( \mathcal K^p \)
  is (fully) faithful if and only if \( \epsilon \) is a pointwise (regular)
  epimorphism (in \( \cat B \comma y \)), which happens if and only if \(p\)
  is a universal (regular) epimorphism in \( \cat B \comma y \), as \(
  \epsilon \) is given pointwise by pullback of \(p\).

  Since \( \cat B \) has a terminal object, the forgetful functor \( \cat B/y
  \to \cat B \) has a right adjoint, hence it preserves colimits.
\end{proof}

Thus, once we have a pullback-stable regular epimorphism \(p\), it is natural
to take an interest in studying the image of \(\mathcal K^p\). To do so, we
make the following elementary observation. Since we have defined descent data
as algebras, we restrict our attention to this context. It should be noted,
however, that the result holds in much more general contexts, and hence its
applicability in descent arguments does not depend on the Bénabou-Roubaud
theorem.

\begin{lemma}
  \label{lem:ap-image-Eilenberg-Moore}
  Let \( ( L\adj U, \epsilon , \eta) \colon \cat A \to \cat B \) be an
  adjunction and let \( T \) be the induced monad. An algebra \( (a, \gamma)
  \) is in the image of the Eilenberg-Moore comparison \( \mathcal K^T \colon
  \cat A \to T \dash \Alg \) if, and only if, \( a \) is in the image of \( U
  \) and
  \begin{equation}
    \label{eq:condition-EM-Image} 
    \epsilon_w \circ L\gamma = \epsilon_w \circ \epsilon_{LUw}.
  \end{equation} 
  where \(w\) is an object such that \( a = Uw \).
\end{lemma}
\begin{proof}
  The algebra \( \mathcal K^Tw \) satisfies \eqref{eq:condition-EM-Image} by
  naturality. Conversely, if an algebra \( (Uw, \gamma) \) satisfies
  \eqref{eq:condition-EM-Image}, then
  \begin{equation*}
    \gamma = U\epsilon_w \circ \eta_{Uw} \circ \gamma
           = U\epsilon_w \circ UL\gamma \circ \eta_{ULUw}
           = U\epsilon_w \circ U\epsilon_{LUw} \circ \eta_{ULUw}
           = U\epsilon_w
  \end{equation*}
  Hence \( (Uw, \gamma) = \mathcal K^Tw \).
\end{proof}

As a corollary, we get a fairly commonly used result in proofs about effective
descent morphisms. It has been, sometimes, implicitly assumed in the
literature. The instance of Le Creurer's argument in Proposition 3.2.4, where
he implicitly uses this result, is of particular interest for our work. 

\begin{corollary}
  \label{Lemma11}
  \( \mathcal K^p \) is essentially surjective if and only if, for all descent
  data \((a,\gamma)\), there is \( f \) such that \( p^*f \cong a \) and \(
  \epsilon_f \circ \gamma = \epsilon_f \circ \epsilon_{p\circ a} \).
\end{corollary}

We finish this section recalling the following classical descent result (see
\cite[2.7]{JT94}, \cite[3.9]{JST04}):

\begin{proposition}
  \label{subcat.descent}
  Let \( U \colon \cat C \to \cat D \) be a fully faithful,
  pullback-preserving functor, and let \(p\) be a morphism in \( \cat C \)
  such that \(Up\) is effective for descent. Then \(p\) is effective for
  descent if and only if for all pullback diagrams of the form
  \begin{equation}
    \label{obs}
    \begin{tikzcd}
      Ux \ar[d] \ar[r] & z \ar[d,"f"] \\
      Ue \ar[r,swap,"Up"] & Ub
    \end{tikzcd}
  \end{equation}
  there exists an isomorphism \( Uy \iso z \) for \(y\) an object of \( \cat
  C \).
\end{proposition}

The following consequence is of particular interest:

\begin{corollary}
  \label{mmc}
  Let \( U \colon \cat C \to \cat D \) be a fully faithful,
  pullback-preserving functor. If there exists \( z \iso Uy \) whenever there
  is an effective descent morphism \( g \colon Ux \to z \), then \( U \)
  reflects effective descent morphisms.
\end{corollary}

\begin{proof}
  Suppose \eqref{obs} is a pullback square. If \( Up
  \) is an effective descent morphism, then so is \( f^*(Up) \colon Ux \to z
  \) by pullback-stability. By hypothesis, we have \( z \iso Uy \), whence we
  conclude that $p$ is effective for descent by Proposition
  \ref{subcat.descent}. 
\end{proof}

  \section{Multicategories and pseudo-equalizers}
    \label{Multicategories and pseudoequalizers} 
    Recall that a monad \(T = (T,m,e) \) is \textit{cartesian} if \(T\) preserves
pullbacks and the naturality squares of \(m\) and \(e\) are pullbacks.

As defined in \cite{Her00}, for \(T\) a cartesian monad on a category \( \cat
B \) with pullbacks, a \(T\)-multicategory internal to \( \cat B \) is a monad
in the bicategory \( \Span_T(\cat B) \), and a functor between two such
\(T\)-multicategories is a monad morphism considering the usual proarrow
equipment \( \cat B\to \Span_T(\cat B) \); these define the category \(
\Cat(T, \cat B) \).  Explicitly, a \(T\)-multicategory is given by an object
\( x_0 \) of \( \cat B \), together with a span 
\begin{equation*}
  \begin{tikzcd}
    Tx_0 & x_1 \ar[l,"d_1",swap] \ar[r,"d_0"] & x_0
  \end{tikzcd}
\end{equation*}
and two morphisms, given by dashed arrows below
\begin{equation*}
  \begin{tikzcd}
    & x_2 \ar[ld,"m \circ Td_1 \circ d_2",swap] 
          \ar[rd,"d_0 \circ d_0"] 
          \ar[dd,"d_1",dashed]
    &&& x_0 \ar[ld,"e",swap] 
           \ar[rd,"\mathrm{id}_{x_0}"] 
           \ar[dd,"s_0",dashed] \\
    Tx_0 && x_0 & Tx_0 && x_0 \\
    & x_1 \ar[lu,"d_1"] \ar[ru,"d_0",swap] 
    &&& x_1 \ar[lu,"d_1"] \ar[ru,"d_0",swap]
  \end{tikzcd}
\end{equation*}
which make the triangles commute, where
\begin{equation*}
  \begin{tikzcd}
    x_2 \ar[d,swap,"d_0"] \ar[r,"d_2"]
    & Tx_1 \ar[d,"Td_0"] \\
    x_1 \ar[r,swap,"d_1"] & Tx_0
  \end{tikzcd}
\end{equation*}
is a pullback diagram. Moreover, this data is required to satisfy certain
identity and associativity conditions, which we will proceed to specify.

Following the terminology of Section \ref{Preliminaries}, we say that a pair
\( g \colon a \to x_1 \), \( f \colon a \to Tx_1 \) is \textit{composable} if
\( d_1g = (Td_0)f \), we write \( (g,f) \colon a \to x_2 \) for the uniquely
defined morphism, and we let \( g \bullet f = d_1(g,f) \). Likewise, define \(
k \bullet_T f = (Td_1)(k,h) \) for \( k \colon a \to Tx_1 \) and \( h \colon a
\to TTx_1 \) such that \( (Td_1)k = (TTd_0)h \) (\(T\)-composable).

The identity properties of the monad guarantee that \( 1_{d_0f} \bullet (e
\circ f) = f = f \bullet 1_{d_1f} \), and the associativity property
guarantees that \( h \bullet (g \bullet_T f) = (h \bullet g) \bullet (m \circ
f) \), where we are implicitly given the following pullback diagram
\begin{equation*}
  \begin{tikzcd}
    x_3 \ar[d,swap, "d_0"] \ar[r,"d_3"]
    & Tx_2 \ar[d,"Td_0"] \\
    x_2 \ar[r,swap,"d_2"] & Tx_1
  \end{tikzcd}
\end{equation*}
for \( h \colon a \to x_1 \), \( g \colon a \to Tx_1 \) and \( f \colon a \to
TTx_1 \) such that \( h,\,g \) are composable and \( g,\,f \) are
\(T\)-composable. Moreover, a functor \( p \colon x \to y \) between
internal \(T\)-multicategories is given by a pair of morphisms \( p_0 \colon
x_0 \to y_0 \) and \( p_1 \colon x_1 \to y_1 \) such that \( d_i \circ p_1 =
(T^ip_0) \circ d_i \) for \(i=0,1\), \( 1_{p_0}  = p_1 1 \) and \( p_1g
\bullet p_1f = p_1(g\bullet f) \). 

Going back to an internal description, we may denote
\begin{itemize}[noitemsep]
  \renewcommand\labelitemi{--}
  \item
    \( s_0 = (\id,\, Ts_0 \circ d_1) \colon x_1 \to x_2 \),

  \item
    \( s_1 = (s_0 \circ d_0,\, e) \colon x_1 \to x_2 \),

  \item
    \( d_1 = (d_0 \circ d_0,\, Td_1 \circ d_3) \),

  \item
    \( d_2 = (d_1 \circ d_0,\, m \circ Td_2 \circ d_3) \),
\end{itemize}
so the above data can be organized in the following diagram
\begin{equation*}
  \begin{tikzcd}[row sep=large, column sep=large]
    x_0 \ar[r,shift left,"s_0"]
        \ar[rd,"e",swap]
    & x_1 \ar[r,shift left=0.6em,"s_0" description,near start]
          \ar[r,shift left=1.2em,"s_1",near start]
          \ar[l,"d_0",shift left]
          \ar[d,"d_1",swap]
          \ar[rd,"e",swap] 
    & x_2 \ar[l,"d_0" description,swap,near start]
          \ar[l,shift left=0.6em,"d_1",near start]
          \ar[d,"d_2",swap]
    & x_3 \ar[l,shift left=0.6em,"d_2"]
          \ar[l,"d_1" description]
          \ar[l,shift right=0.6em,"d_0",swap] 
          \ar[d,"d_3",swap] \\
    & Tx_0 \ar[r,"Ts_0",shift left]
    & Tx_1 \ar[d,"Td_1",swap]
           \ar[l,"Td_0",shift left]
    & Tx_2 \ar[l,shift right,"Td_0",swap]
           \ar[l,shift left,"Td_1"] 
           \ar[d,"Td_2",swap] \\
    && TTx_0  \ar[lu,"m"]
    & TTx_1 \ar[lu,"m"] \ar[l,"TTd_0"]
  \end{tikzcd}
\end{equation*}
which is similar to \cite[Figure 1]{Bur71}. In fact, one may define
\(T\)-multicategory as a diagram satisfying certain relations, a description
particularly suitable for our techniques in Section \ref{First approach}.
First, we let \(\cat S\) be the (finite limit) sketch given by the following
graph
\begin{equation}
  \label{almost.cosimp}
  \begin{tikzcd}[row sep=large, column sep=large]
    x_0 \ar[r,shift left,"s_0"]
        \ar[rd,"e_0",swap]
    & x_1 \ar[r,shift left=0.6em,"s_0" description,near start]
          \ar[r,shift left=1.2em,"s_1",near start]
          \ar[l,"d_0",shift left]
          \ar[d,"d_1",swap]
          \ar[rd,"e_1",swap] 
    & x_2 \ar[l,"d_0" description,swap,near start]
          \ar[l,shift left=0.6em,"d_1",near start]
          \ar[d,"d_2",swap]
    & x_3 \ar[l,shift left=0.6em,"d_2"]
          \ar[l,"d_1" description]
          \ar[l,shift right=0.6em,"d_0",swap] 
          \ar[d,"d_3",swap] \\
    & x'_0 \ar[r,"s'_0",shift left]
    & x'_1 \ar[d,"d'_1",swap]
           \ar[l,"d'_0",shift left]
    & x'_2 \ar[l,shift right,"d'_0",swap]
           \ar[l,shift left,"d'_1"] 
           \ar[d,"d'_2",swap] \\
    && x''_0  \ar[lu,"m_0"]
    & x''_1 \ar[lu,"m_1"] \ar[l,"d_0''"]
  \end{tikzcd}
\end{equation}
with relations resembling cosimplicial identities
\begin{itemize}[noitemsep]
  \renewcommand\labelitemi{--}
  \item
    \( s_1 \circ s_0 = s_0 \circ s_0 \colon x_0 \to x_2 \),
  \item
    \( d_{1+i} \circ s_i = e_i \colon x_i \to x'_i \), 
  \item
    \( d_i \circ s_j = \id \colon x_i \to x_i \), 
  \item
    \( d_2 \circ s_0 = s_0' \circ d_1 \colon x_1 \to x'_1 \),
  \item
    \( d_0 \circ s_1 = s_0 \circ d_0 \colon x_1 \to x_1 \),
  \item
    \( d'_0 \circ s'_0 = \id \colon x'_0 \to x'_0 \),
  \item
    \( d_{1+i} \circ d_{1+i} = m_i \circ d'_{1+i} \circ d_{2+i} \colon x_{2+i}
    \to x'_i\), 
  \item
    \( d_{1+i} \circ d_0 = d_0' \circ d_{2+i} \colon x_{2+i} \to x_i \),  
  \item
    \( d'_j \circ d_{2+i} = d_{1+i} \circ d_j \colon x_{2+i} \to x'_i \), 
  \item
    \( d_0 \circ d_1 = d_0 \circ d_0 \colon x_2 \to x_0 \),
  \item
    \( d_j \circ d_{1+i} = d_i \circ d_j \colon x_3 \to x_1 \),
  \item
    \( d'_1 \circ d_0' = d_0'' \circ d_2' \colon x_2' \to x_0'' \),
  \item
    \( d'_0 \circ d'_1 = d'_0 \circ d'_0 \colon x'_2 \to x_0' \),
\end{itemize}
and limit cones
\begin{equation}
  \label{pb.squares}
  \begin{tikzcd}
    x_{2+i} \ar[r,"d_0"] \ar[d,"d_{2+i}"] & x_{1+i} \ar[d,"d_{1+i}"] \\
    x'_{1+i} \ar[r,"d'_0"] & x'_i
  \end{tikzcd}
  \qquad
  \begin{tikzcd}
    x_2' \ar[r,"d'_0"] \ar[d,"d'_2"] & x_1' \ar[d,"d'_1"] \\
    x''_1 \ar[r,"d''_0"] & x''_0
  \end{tikzcd}
\end{equation}
with \( i=0,1 \) and \( j \leq i \). Abusing notation, we also denote by \(
\cat S \) the category generated by the graph (\ref{almost.cosimp}) and the
given relations. Writing \( \Mod(\cat S, \cat B) \) for the category of \(
\cat B \)-models of \( \cat S \), we have:

\begin{lemma}
  \label{multicat.as.eq}
  For a cartesian monad \( (T,m,e) \) on a category \( \cat B \) with
  pullbacks, \( \Cat(T,B) \) is given as the equalizer of the following
  composite of pullback-preserving functors:
  \begin{equation}
    \label{bieq}
    \begin{tikzcd}[cramped,column sep=small]
      \Mod(\cat S, \cat B) \ar[r,"I"]
      & \left[ \cat S, \cat B \right] \ar[r,shift left,"S^*_-"] 
                                    \ar[r,shift right,"\Phi",swap]
      & \left[\cat S_T, \cat B\right] 
        \times \left[\cat S_{m_0}, \cat B\right]
        \times \left[\cat S_{m_1}, \cat B\right]
        \times \left[\cat S_{e_0}, \cat B\right]
        \times \left[\cat S_{e_1}, \cat B\right]
    \end{tikzcd}
  \end{equation}
  Moreover, \( \Cat(T, \cat B) \) has pullbacks and the canonical functor \(
  \Cat(T,\cat B) \to \Mod(\cat S, \cat B) \) preserves them.
\end{lemma}

\begin{remark}
  It might seem superfluous to require the right diagram of (\ref{pb.squares})
  to be a pullback, as the equalizer condition will force \( x'_i = Tx_i \)
  and \( x''_i = TTx_i \), and since \(T\) preserves pullbacks, the pullback
  condition for the aforementioned diagram is already guaranteed.

  Moreover, omitting this apparently redundant diagram, an analogous version
  of Lemma 3.1 would describe Burroni's notion of \(T\)-multicategories
  (check \cite{Bur71}, where this extra pullback condition is \textit{not}
  required), even when \(T\) is not cartesian, or even pullback-preserving.

  In spite of the above reasons, this requirement is justified by the sharper
  results we obtain about effective descent in \( \Mod(\cat S, \cat B) \) (see
  Proposition \ref{mod.descent}), and consequently, in \( \Cat(T, \cat B) \)
  as well (see Theorem \ref{first.theorem}). 
\end{remark}

Note that the inclusion \( \Mod(\cat S, \cat B) \to [\cat S, \cat B] \)
is an iso-inserter of categories of diagrams, thus it creates limits. 

The categories \( \cat S_I \), \( \cat S_T \), \( \cat S_{m_i} \), and \( \cat
S_{e_i} \) for \(i=0,1\) are subcategories of \( \cat S \), respectively given
by
\begin{equation*}
  \begin{tikzcd}
    x_0 \ar[r,shift left,"s_0"]
    & x_1 \ar[l,"d_0",shift left]
          \ar[d,"d_1",swap]
    & x_2 \ar[l,shift right,"d_0",swap]
          \ar[l,shift left,"d_1"] 
          \ar[d,"d_2"] \\
    & x'_0   
    & x'_1 \ar[l,"d'_0"]
  \end{tikzcd}
  \qquad
  \begin{tikzcd}
    x'_0   \ar[r,shift left,"s'_0"]
    & x'_1 \ar[d,"d'_1",swap]
           \ar[l,"d'_0",shift left]
    & x'_2 \ar[l,shift right,"d'_0",swap]
           \ar[l,shift left,"d'_1"] 
          \ar[d,"d'_2"] \\
    & x''_0   
    & x''_1 \ar[l,"d''_0"]
  \end{tikzcd}
\end{equation*}
\begin{equation*}
  \begin{tikzcd}
    x''_i \ar[r,"m_i"] & x'_i
  \end{tikzcd}
  \qquad
  \begin{tikzcd}
    x_i \ar[r,"e_i"] & x'_i
  \end{tikzcd}
\end{equation*}
and write \( S^*_I \), \( S^*_T \), \( S^*_{m_i} \), \( S^*_{e_i} \), 
for the restriction functors. Also write \( x_0^* \) and \( x_1^* \colon [\cat
S,\cat B] \to \cat B \) for the projections. With these, \( S^*_- \), and \(
\Phi \) are the uniquely determined functors given by the following
\begin{equation*}
  \begin{tikzcd}
    \left[\cat S, \cat B\right] \ar[r,"S_T^*"] 
                                \ar[d,"S_I^*",swap]
    & \left[\cat S_T, \cat B\right] \\
    \left[\cat S_I, \cat B\right] \ar[r,"\iso"]
    & \left[\cat S_T, \cat B\right] \ar[u,"T_*",swap]
  \end{tikzcd}
\end{equation*}
\begin{equation*}
  \begin{tikzcd}
    \left[\cat S, \cat B\right] \ar[rr,"S_{m_i}^*"] 
                                \ar[rd,"x_i^*",swap]
    && \left[\cat S_{m_i}, \cat B\right] \\
    & \cat B \ar[ru,"\hat m",swap]
  \end{tikzcd}
  \qquad
  \begin{tikzcd}
    \left[\cat S, \cat B\right] \ar[rr,"S_{e_i}^*"] 
                                \ar[rd,"x_i^*",swap]
    && \left[\cat S_{e_i}, \cat B\right] \\
    & \cat B \ar[ru,"\hat e",swap]
  \end{tikzcd}
\end{equation*}
where \(T_*\), \( \hat m \) and \( \hat e \) are the functors induced by the
monad \( T \). Note that these preserve pullbacks exactly when \( T \) is
cartesian. 

Note that, in general, the equalizer is a full subcategory of the
pseudo-equalizer: for functors \(F, G \colon \cat C \to \cat D\), the
category \( \PsEq(F,G) \) is the category whose objects are pairs \( (c,\phi)
\) where \(c\) is an object of \( \cat C \) and \( \iota \colon Fy \to Gy \)
is an isomorphism, and morphisms \( (c,\phi) \to (d,\psi) \) are morphisms \(
f \colon c \to d \) such that \( Gf \circ \phi = \psi \circ Ff \). Thus, the
full embedding may be given on objects by \( x \mapsto (x,\id) \). 

Henceforth, we denote
\begin{equation}
  \label{this.is.P}
  \cat P = \PsEq(S^*_- \circ I, \Phi \circ I).
\end{equation}

\begin{lemma}
  \label{embedding}
  The inclusion \( \Cat(T,\cat B) \to \cat P \) is full and preserves
  pullbacks.
\end{lemma}

\begin{proof}
  The inclusion \( \Cat(T,\cat B) \to \Mod(\cat S, \cat B) \) preserves
  pullbacks, which are then created by \( \cat P \to \Mod(\cat S, \cat B) \).
\end{proof}

Given an object \((y,\iota)\) of \( \cat P \), \(\iota\) can be explicitly
described as a family of isomorphisms making the appropriate squares commute:
\begin{equation*}
  \begin{tikzcd}[column sep=large]
    y'_0   \ar[r,"s'_0",bend left]
          \ar[ddd,swap,"\iota^T_0"]
    & y'_1 \ar[l,"d'_0",bend left,swap]
          \ar[d,"d'_1"]
    & y'_2 \ar[l,"d'_1",shift left,near start] 
          \ar[l,"d'_0",shift right,swap,near start] 
          \ar[ddd,bend left=135,"\iota^T_2"] 
          \ar[d,"d'_2"] \\
    & y''_0 \ar[d,"\iota^T_3"]
    & y''_1 \ar[d,"\iota^T_4"] \\
    & Ty'_0 & Ty'_1 \\
    Ty_0 \ar[r,"Ts_0",bend left]
    & Ty_1 \ar[l,"Td_0",bend left]
           \ar[u,"Td_1"]
    & Ty_2 \ar[l,"Td_0",shift right,swap,,near start]
           \ar[l,"Td_1",shift left,near start]
           \ar[u,"Td_2"]
    \ar[ddd,bend left=135,"\iota^T_1",crossing over,from=1-2,to=4-2]
  \end{tikzcd}
\end{equation*}
\begin{equation*}
  \begin{tikzcd}
    y''_i \ar[d,swap,"\iota^{m_i}_0"] 
        \ar[r,"m_i"]
    & y'_i \ar[d,"\iota^{m_i}_1"] \\
    TTy_i \ar[r,"m"]
    & Ty_i
  \end{tikzcd}
  \qquad
  \begin{tikzcd}
    y_i \ar[r,"e_i"] \ar[d,swap,"\iota^{e_i}_0"]
    & y'_i \ar[d,"\iota^{e_i}_1"] \\
    y_i \ar[r,swap,"e"]
    & Ty_i
  \end{tikzcd}
\end{equation*}

\begin{lemma}
  An object \( (y,\iota) \) of \( \cat P \) is isomorphic to a
  \(T\)-multicategory if and only if the following coherence conditions hold:
  \begin{enumerate}[label=\emph{(\roman*)},noitemsep]
    \item
      \( \iota^{m_i}_1 = \iota^{e_i}_1 = \iota^T_i \) for \( i=0,1 \),

    \item
      \( T\iota^T_i \circ \iota^T_{3+i} = \iota^{m_i}_0 \), for \( i=0,1 \),
    
    \item
      \( \iota_0^{e_i} = \id \) for \( i=0,1 \),
  \end{enumerate}
\end{lemma}
Such an object \( (y,\iota) \) satisfying these conditions is said to be
\textit{coherent}.

\begin{proof}
  Given a coherent \( (y,\iota) \), we define a \(T\)-multicategory \( \hat y
  \) such that \( \hat y_0 = y_0 \), \( \hat y_1 = y_1 \), and we consider
  the span
  \begin{equation*}
    \begin{tikzcd}
      Ty_0 & y_1 \ar[l,"\iota^T_0 \circ d_1",swap]
                \ar[r,"d_0"]
           & y_0,
    \end{tikzcd}
  \end{equation*}
  so that we have \( \hat d_1 = \iota^T_0 \circ d_1 \), \( \hat d_0 = d_0 \),
  and we let \( \hat d_1 = d_1 \colon x_2 \to x_1 \) and \( \hat s_0 = s_0
  \colon x_0 \to x_1 \). 

  Consider the diagram for \( i=0,1 \):
  \begin{equation*}
    \begin{tikzcd}
      y_{2+i} \ar[d,swap,"d_0"]
              \ar[r,"d_{2+i}"]
      & y'_{1+i} \ar[d,"d'_0"]
                 \ar[r,"\iota^T_{1+i}"]
      & Ty_{1+i} \ar[d,"Td_0"] \\
      y_{1+i} \ar[r,swap,"d_{1+i}"]
      & y'_i \ar[r,swap,"\iota^T_i"]
      & Ty_i
    \end{tikzcd}
  \end{equation*}
  The right square is a pullback because \( \iota^T_j \) is an isomorphism
  for \( j=0,1,2 \), and the left square is a pullback by definition,
  therefore the outer rectangle is a pullback as well. 

  Let \( \hat d_0 = d_0 \colon y_{2+i} \to y_{1+i} \) and \( \hat d_{2+i} =
  \iota^T_{1+i} \circ d_{2+i} \colon y_{2+i} \to y_{1+i} \) for \( i = 0,1
  \). We claim that every triangle commutes:
  \vspace{-1cm}
  \begin{multicols}{2}
    \begin{equation}
      \label{left}
      \begin{tikzcd}
        & y_2 \ar[ld,swap,"m \circ T\hat d_1 \circ \hat d_2"]
              \ar[rd,"d_0 \circ d_0"]
              \ar[dd,"d_1"] \\
        Ty_0 && y_0 \\
        & y_1 \ar[ul,"\hat d_1"]
              \ar[ur,"d_0",swap]
      \end{tikzcd}
    \end{equation} \break
    \begin{equation}
      \label{right}
      \begin{tikzcd}
        & y_0 \ar[ld,swap,"e"]
              \ar[rd,equals]
              \ar[dd,"s_0"] \\
        Ty_0 && y_0 \\
        & y_1 \ar[ul,"\hat d_1"]
              \ar[ur,"d_0",swap]
      \end{tikzcd}
    \end{equation}
  \end{multicols}
  Of course, both right triangles commute by definition. Moreover, we have
  that the diagram
  \begin{equation}
    \label{rect.pb}
    \begin{tikzcd}
      y_2 \ar[r,"d_2"]
      & y'_1 \ar[r,"d'_1"] \ar[d,swap,"\iota^T_1"]
      & y''_0 \ar[r,"M"] \ar[rd,"\iota^m_0"] \ar[d,"\iota^T_3"]
      & y'_0 \ar[rd,"\iota^m_1"] \\
      & Ty_1 \ar[r,swap,"Td_1"] 
      & Ty'_0 \ar[r,swap,"T\iota^T_0"]
      & TTy_0 \ar[r,swap,"m"]
      & Ty_0
    \end{tikzcd}
  \end{equation}
  commutes by the naturality of \( \iota \) and coherence of \( (y,\iota) \).

  Since \( M \circ d'_1 \circ d_2 = d_1 \circ d_1 \) by definition, the 
  left triangle of (\ref{left}) commutes. The left triangle of (\ref{right})
  also commutes, for we have \( e = \iota^{e_0}_1 \circ e_0 \),
  \(\iota^{e_0}_1 = \iota^T_0 \) and \(e_0 = d_1 \circ s_0 \).

  We claim it is possible to define
  \begin{itemize}[noitemsep]
    \renewcommand\labelitemi{--}
    \item
      \( \hat s_0 = (\id, Ts_0 \circ \hat d_1) \)

    \item
      \( \hat s_1 = (s_0 \circ d_0, e) \)

    \item
      \( \hat d_2 = (d_1 \circ d_0, m \circ T\hat d_2 \circ \hat d_3) \)

    \item
      \( \hat d_1 = (d_0 \circ d_0, Td_1 \circ \hat d_3) \)
  \end{itemize}
  and in order to verify our claim, we must show that
  \begin{itemize}[noitemsep]
    \renewcommand\labelitemi{--}
    \item
      \( \hat d_1 = Td_0 \circ Ts_0 \circ \hat d_1 \),
    \item
      \( \hat d_1 \circ s_0 \circ d_0 = Td_0 \circ e \),
    \item
      \( \hat d_1 \circ d_1 \circ d_0  
          = Td_0 \circ m \circ T\hat d_2 \circ \hat d_3 \),
    \item
      \( \hat d_1 \circ d_0 \circ d_0 = Td_0 \circ Td_1 \circ \hat d_3 \),
  \end{itemize}
  Since \(d_0 \circ s_0 \) is the identity, the first equation is satisfied.
  We have
  \begin{equation*}
    \iota^T_0 \circ d_1 \circ s_0 \circ d_0 
      = \iota^{e_0}_1 \circ e_0 \circ d_0
      = e \circ d_0 = Td_0 \circ e,
  \end{equation*}
  which verifies the second. For the third and fourth, we have
  \begin{align*}
    \iota^T_0 \circ d_1 \circ d_0 \circ d_0
      &= \iota^T_0 \circ d'_0 \circ d_2 \circ d_0 \\
      &= Td_0 \circ \iota^T_1 \circ d'_0 \circ d_3 \\
      &= Td_0 \circ Td_0 \circ \iota^T_2 \circ d_3 \\
      &= Td_0 \circ Td_1 \circ \hat d_3, \\
    Td_0 \circ m \circ T\iota^T_1 \circ Td_2 \circ \iota^T_2 \circ d_3
      &= m \circ TTd_0 \circ T\iota^T_1 \circ Td_2 \circ \iota^T_2 \circ d_3\\
      &= m \circ T\iota^T_0 \circ Td'_0 \circ Td_2 \circ \iota^T_2 \circ d_3\\
      &= m \circ T\iota^T_0 \circ Td_1 \circ Td_0 \circ \iota^T_2 \circ d_3\\
      &= m \circ T\iota^T_0 \circ Td_1 \circ \iota^T_1 \circ d'_0 \circ d_3\\
      &= m \circ T\iota^T_0 \circ \iota^T_3 \circ d'_1 \circ d_2 \circ d_0\\
      &= m \circ \iota^{m_0}_0 \circ d'_1 \circ d_2 \circ d_0 \\
      &= \iota^{m_0}_1 \circ m_0 \circ d'_1 \circ d_2 \circ d_0 \\
      &= \iota^{m_0}_1 \circ d_1 \circ d_1 \circ d_0,
  \end{align*}
  as desired. 

  Recalling that the left square in (\ref{rect.pb}) is a pullback for \(i=0,
  1\), it follows that \( s_0, s_1 \colon x_1 \to x_2 \) and \( d_1, d_2
  \colon x_3 \to x_2 \) are given by \( (\id, s'_0 \circ d_1) \), \( (s_0
  \circ d_0, e_1) \), \( (d_1 \circ d_0, m_1 \circ d'_2 \circ d_3) \) and \(
  (d_0 \circ d_0, d'_1 \circ d_3) \), respectively. But these are just \( \hat
  s_0, \hat s_1, \hat d_1, \hat d_2 \), respectively. 

  The converse is implied by the result that follows.
\end{proof}

\begin{theorem}
  \label{teorema.muito.bonito}
  If \( f \colon (x,\id) \to (y,\iota) \) is a pointwise epimorphism in \(
  \cat P \), then \( (y,\iota) \) is coherent. Hence, \( (y, \iota) \) is
  isomorphic to a \(T\)-multicategory.
\end{theorem}

\begin{proof}
  A morphism \( f \colon (x,\id) \to (y,\iota) \) is a morphism \( f \colon x
  \to y \) such that \( Gf = \iota \circ Ff \), which translates to the
  following equations:
  \begingroup
  \allowdisplaybreaks
  \begin{align*}
    f_0 &= \iota^{e_0}_0 \circ f_0  && 
    &f'_0 = \iota^T_0 \circ Tf_0 
          = \iota^{m_0}_1 \circ Tf_0
          = \iota^{e_0}_1 \circ Tf_0 \\
    f_1 &= \iota^{e_1}_0 \circ f_1  && 
    &f'_1 = \iota^T_1 \circ Tf_1 
          = \iota^{m_1}_1 \circ Tf_1 
          = \iota^{e_1}_1 \circ Tf_1 \\
    &&& &f'_2 = \iota^T_2 \circ Tf_2 \\
    f''_0 &= \iota^T_3 \circ Tf'_0 
           = \iota^{m_0}_0 \circ TTf_0 \\
    f''_1 &= \iota^T_4 \circ Tf'_1 
           = \iota^{m_1}_0 \circ TTf_1 
  \end{align*}
  \endgroup
  and noting that \( f_i, f'_i, f''_i \) all are epimorphisms for all \(i\)
  we recover the coherences; just note that \(Tf_i\) and \(TTf_i\) are
  epimorphisms as well, and that \( Tf'_i = T\iota^T_i \circ TTf_i \).
\end{proof}

  \section{Descent via bilimits}
    \label{First approach}
    Recall the pseudo-equalizer \( \cat P \) defined in (\ref{this.is.P}) from
the previous section. We understand the effective descent morphisms of \( \cat
P \) via the effective descent morphisms of \( \Mod(\cat S, \cat B) \) by the
following instance of \cite[Theorem 9.2]{Luc18}. 

\begin{proposition}
  \label{pseq.descent}
  Suppose that we have a pseudo-equalizer of categories and
  pullback-preserving functors
  \begin{equation*}
    \begin{tikzcd}
      \PsEq(F,G) \ar[r,"I"]
      & \cat C \ar[r,shift left,"F"] \ar[r,shift right,"G",swap]
      & \cat D
    \end{tikzcd}
  \end{equation*}
  and let \(f\) be a morphism in the pseudo-equalizer. Then \(f\) is effective
  for descent whenever \(If\) is effective for descent and \( FIf \iso GIf \)
  is a pullback-stable regular epimorphism.
\end{proposition}

Furthermore, by the work of \cite{Cre99}, we are able to provide sufficient
conditions for effective descent in \( \Mod(\cat S, \cat B) \) for \( \cat B
\) with finite limits:

\begin{proposition}
  \label{mod.descent}
  If a morphism \(p\) in \( \Mod(\cat S, \cat B) \) is such that
  \begin{itemize}[noitemsep]
    \renewcommand\labelitemi{--}
    \item
      \( p_0, p_1, p'_0, p'_1, p''_0, p''_1 \) are effective descent
      morphisms in \( \cat B \),
    \item
      \( p_2, p'_2 \) are descent morphisms in \( \cat B \),
    \item
      \( p_3 \) is an almost descent morphism in \( \cat B \),
  \end{itemize}
  then \(p\) is an effective descent morphism in \( \Mod(\cat S, \cat B) \).
\end{proposition}

\begin{proof}
  We refer the reader to Section 3.2 \textit{ibid} if they wish to fill in the
  details. The sketch \( \cat S \) may be given as an essentially algebraic
  theory with sorts \( x_0 \), \(x_1\), \(x_0'\), \(x_1'\), \(x_0''\), \(
  x_1'' \), partially defined operations \( d_1 \colon x_1 \times x_1' \to x_1
  \), \( d_1' \colon x'_1 \times x''_1 \to x'_1 \), and equation \(
  (d_1 \circ (\id,d_1), d_1 \circ (d_1,\id)) \colon x_1 \times x'_1 \times
  x''_1 \to x_1 \), among other data and equations. Then apply Proposition
  3.2.4 \textit{ibid}.
\end{proof}

With \(T\) cartesian, diagram (\ref{bieq}) is a pseudo-equalizer, so we are
under the hypothesis of Proposition 4.1. Therefore:

\begin{lemma}
  \label{part.1}
  A morphism \(p\) in \( \cat P \) is effective for descent whenever \(p\) is
  effective for descent in \( \Mod(\cat S, \cat B) \) and \( S^*_Xp \) is a
  descent morphism for each \( X = T, m_0, m_1, e_0, e_1 \).

  In particular, if \( p \) satisfies the conditions in Proposition
  \ref{mod.descent}, then \( p \) is effective for descent in the
  pseudo-equalizer.
\end{lemma}

\begin{proof}
  We observe that a morphism in a product of categories is of descent if and
  only if each component is a descent morphism. Moreover, pointwise
  (effective) descent in \( [\cat S, \cat B] \) implies pointwise descent in
  \( [\cat S_X, \cat B] \) for every \(X\). Therefore, the result follows by
  Proposition \ref{pseq.descent}. 
\end{proof}

By Lemma \ref{embedding}, we may apply the previous proposition to \( U
\colon \Cat(T, \cat B) \to \cat P \). Consequently, we can show that:

\begin{lemma}
  \label{part.2} 
  \( U \) reflects effective descent morphisms.
\end{lemma}

\begin{proof}
  Since every effective descent morphism is an epimorphism, the result follows
  by Theorem \ref{teorema.muito.bonito} and Corollary \ref{mmc}.
\end{proof}

Combining Lemmas \ref{part.1} and \ref{part.2}, we get our main result:

\begin{theorem}
  \label{first.theorem}
  For \( \cat B \) with finite limits, let \( p \colon x \to z \) be a
  \(T\)-multicategory functor internal to \( \cat B \). If \( Tp_1 \) is an
  effective descent morphism, \( Tp_2 \) is a descent morphism and \( p_3 \)
  is an almost descent morphism in \( \cat B \), then \(p\) is an effective
  descent morphism in \( \Cat(T,\cat V) \).
\end{theorem}
\begin{proof}
  By the results in Appendix \ref{Appendix}, (observe that \( Tp_1 \) is a
  \(T\)-graph morphism), we guarantee that \(p\) is an effective descent
  morphism in \( \Mod(\cat S, \cat B) \). Now apply Theorem \ref{part.2}.
\end{proof}

  \section{Descent via sketches}
    \label{Second approach}
    In this section, we extend the techniques of \cite[Chapter 3]{Cre99} to give
refined sufficient conditions for (effective) descent morphisms in \(
\Cat(T,\cat B) \) in the broader sense of Burroni; that is, without requiring
\(T\) to be cartesian (though we require \(T\) to preserve kernel pairs for
Theorem \ref{main.theorem.2}), while keeping the definition of
\(T\)-multicategory intact. We highlight that given a functor \( p \colon x
\to y \) of internal multicategories, if \( p_1 \) is a pullback-stable
(regular) epimorphism, or of effective descent, then so is \( p_0 \) by Lemma
\ref{lemmaB3}.

\begin{lemma}
  Let \(p \colon x \to y\) be a functor of internal \(T\)-multicategories.
  If \(p_1\) is an (pullback-stable) epimorphism in \( \cat B \), then so is
  \(p\) in \( \Cat(T,\cat B) \).
\end{lemma}

\begin{proof}
  Given functors \(q,r\) such that \( qp = rp \), we have \( q_ip_i = r_ip_i
  \), and therefore \( q_i = r_i \) for \( i = 0,\, 1\), hence \( q = r \),
  thus \(p\) is an epimorphism. Since pullbacks are calculated pointwise, \( p
  \) must be pullback-stable whenever \(p_1\) is.
\end{proof}

\begin{lemma}
  Let \(p\) be a functor of internal \(T\)-multicategories. If 
  \begin{itemize}
    \renewcommand\labelitemi{--}
    \item
      \(p_1\) is a (pullback-stable) regular epimorphism in \( \cat B \), 
    \item
      \(p_2\) is an (pullback-stable) epimorphism in \( \cat B \),
  \end{itemize}
  then \(p\) is a (pullback-stable) regular epimorphism in \( \Cat(T,\cat B)
  \).
\end{lemma}

\begin{proof}
  Consider the kernel pair \( r,s \) of \( p \), and let \( q \colon x \to z
  \) be a functor such that \( q \circ r = q \circ s \). Then there exist
  unique morphisms \(k_0\), \(k_1\) such that \( k_ip_i = q_i \) for \( i =
  0,\, 1 \). We claim these morphisms define a functor \( y \to z \). We have
  \begin{align}
    d_1 \circ k_1 \circ p_1 &= d_1 \circ q_1 
                             = Tq_0 \circ d_1 
                             = Tk_0 \circ Tp_0 \circ d_1 
                             = Tk_0 \circ d_1 \circ p_1 \\
    d_0 \circ k_1 \circ p_1 &= d_0 \circ q_1 
                             = q_0 \circ d_0 
                             = k_0 \circ p_0 \circ d_0 
                             = k_0 \circ d_0 \circ p_1 \\
    k_1 \circ d_1 \circ p_2 &= k_1 \circ p_1 \circ d_1 
                             = q_1 \circ d_1 
                             = d_1 \circ q_2 
                             = d_1 \circ k_2 \circ p_2,
  \end{align}
  and since \(p_1,p_2\) are epimorphisms, cancellation allows us to conclude
  that \(k\) is a functor (we note that \(k_2\) is defined as \( k_2(g,f) =
  (k_1g,k_1f) \), and hence \( q_2 = k_2p_2 \)).

  Again, pointwise calculation of pullbacks guarantees pullback stability.
\end{proof}

\begin{theorem}
  \label{main.theorem.2}
  Let \(p\) be a functor of internal \(T\)-multicategories, and assume \(T\)
  preserves kernel pairs. If 
  \begin{itemize}
    \renewcommand\labelitemi{--}
    \item 
      \(p_1\) is an effective descent morphism in \( \cat B \),
    \item 
      \(p_2\) is a descent morphism in \( \cat B \), 
    \item
      \(p_3\) is an almost descent morphism in \( \cat B \),
  \end{itemize}
  then \(p\) is effective for descent in \( \Cat(T,\cat B) \).
\end{theorem}

\begin{proof}
  By the previous lemma, and Proposition \ref{uniregepi.is.descent}, the
  comparison functor \( \mathcal K^p \) is fully faithful. Hence, we aim to
  prove that \( \mathcal K^p \) is also essentially surjective under our
  hypotheses, thereby concluding that \(p\) is effective for descent. 

  Suppose we are given a \(p^*p_!\)-algebra \( (a,\gamma) \), where \( a
  \colon v \to x \) is a functor and \(\gamma \colon u \to v \) is the
  algebra structure. We have equivalences \( \mathcal K_i \colon \cat B \comma
  y_i \to \Desc(p_i) \), for \( i=0,\,1 \), and \( (a,\gamma) \) then
  determines algebras \( (a_i,\gamma_i) \) for \( i=0,1 \). Hence, there exist
  \( f_i \colon w_i \to y_i \) and \( h_i \colon v_i \to w_i \) such that the
  following diagram
  \begin{equation*}
    \begin{tikzcd}
      v_i \ar[r,"h_i"] \ar[d,swap,"a_i"] & w_i \ar[d,"f_i"] \\
      x_i \ar[r,swap,"p_i"] & y_i
    \end{tikzcd}
  \end{equation*}
  is a pullback square, and moreover, we have \( h_i\circ \gamma_i = h_i \circ
  \epsilon_{p_i \circ a_i} \). We claim that
  \begin{itemize}[noitemsep]
    \renewcommand\labelitemi{--}
    \item
      \( h_0, h_1 \) determine a functor \( h \colon v \to w \),
    \item
      \( f_0, f_1 \) determine a functor \( f \colon w \to y \),
  \end{itemize}
  so that the above lifts to a pullback diagram of \(T\)-multicategories.

  The hypothesis that \( p_1, p_2 \) are pullback-stable regular epimorphisms
  implies that \( h_1, h_2 \) are regular epimorphisms. Taking kernel
  pairs and noting that \(T\) preserves them, we get
  \begin{equation*}
    \begin{tikzcd}
      u_1 \ar[r,shift right]
          \ar[r,shift left]
          \ar[d,swap,"d_i"]
      & v_1 \ar[r,"h_1"] 
            \ar[d,"d_i"]
      & w_1 \ar[d,"d_i",dashed] \\
      T^i u_0 \ar[r,shift right]
              \ar[r,shift left]
      & T^i v_0 \ar[r,swap,"T^ih_0"]
      & T^i w_0 
    \end{tikzcd}
  \end{equation*}
  \begin{equation*}
    \begin{tikzcd}
      u_0 \ar[r,shift right]
          \ar[r,shift left]
          \ar[d, swap,"s_0"]
      & v_0 \ar[r,"h_0"] 
            \ar[d,"s_0"]
      & w_0 \ar[d,"s_0",dashed] \\
      u_1 \ar[r,shift right]
          \ar[r,shift left]
      & v_1 \ar[r,swap,"h_1"]
      & w_1 
    \end{tikzcd}
  \end{equation*}
  \begin{equation*}
    \begin{tikzcd}
      u_2 \ar[r,shift right]
          \ar[r,shift left]
          \ar[d,swap,"d_1"]
      & v_2 \ar[r,"h_2"] 
            \ar[d,"d_1"]
      & w_2 \ar[d,"d_1",dashed] \\
      u_1 \ar[r,shift right]
          \ar[r,shift left]
      & v_1 \ar[r,swap,"h_1"]
      & w_1
    \end{tikzcd}
  \end{equation*}
  therefore there exist unique morphisms making every right hand side square
  commute. We note that we define \(h_2(g,f) = (h_1g, (Th_1)f) \). Assuming
  that \(w\) is in fact a \(T\)-multicategory, we may already conclude that
  \(h\) is a functor. The hypothesis that \(p_1,p_2,p_3\) are pullback-stable
  epimorphisms implies that \(h_1,h_2,h_3\) are epimorphisms. We have
  equations
  \begin{align*}
    d_1s_0h_0 &= (Th_0)d_is_0 = (Th_0)e = eh_0 \\
    d_0s_0h_0 &= h_0d_0s_0 = h_0 \\
    d_1d_1h_2 &= (Th_0)d_1d_1 = (Th_0)m(Td_1)d_2 = m(Td_1)d_hk_2 \\
    d_0d_1h_2 &= h_0d_0d_1 = h_0d_0d_0 = d_0d_0h_2 \\
    d_1s_ih_1 &= h_1d_1s_i = h_1s_0d_0 = s_0d_0h_1 \\
    d_1d_2h_3 &= h_1d_1d_2 = h_1d_1d_1 = d_1d_1h_3 
  \end{align*}
  and by cancellation, we conclude \(w\) is a \(T\)-multicategory (proving
  our assumption) and, similarly, we can show that \(f\) is a functor, by
  following the same strategy as in the previous lemma. This confirms that
  \(p^*\) is essentially surjective.

  Finally, it is immediate that \( h \circ \gamma = h \circ \epsilon_{p\circ
  a} \), since \( h_i \circ \gamma_i = h_i \circ \epsilon_{p_i \circ a_i} \)
  for \( i = 0,1 \) and pullbacks are calculated pointwise. The result now
  follows by Corollary \ref{Lemma11}.
\end{proof}

  \section{Epilogue}
    \label{Note on examples}
    There are sparse examples of cartesian monads, and therefore sparse examples
of categories of internal multicategories over a monad. For \( \cat B \)
finitely extensive with finite limits and pullback-stable nested countable
unions, as in \cite[Appendix D]{Lei04}, the free category monad on graphs
internal to \( \cat B \) is cartesian, and therefore so is the free monoid
monad \( \word \) on \( \cat B \). In fact, Leinster's construction is
iterable, and most known examples fit into the above conditions. 

A class of examples outside of the previous setting is given by free monoid
monads on extensive categories with finite limits (thus, trading off the
requirement of the aforementioned unions by infinitary extensivity).  These
are also cartesian; the idea is that the coproduct functor \( \Fam(\cat B) \to
\cat B \) preserves finite limits, so we may construct the required limit
diagrams in \( \Fam(\cat B) \), allowing us to conclude that such monads
preserve pullbacks and that the required naturality squares are pullbacks. 

Given a cartesian monad on a category \( \cat B \) with pullbacks and
\(\intcat C\) an internal \(T\)-multicategory, we can construct a cartesian
monad \(T_{\intcat C}\) on \( \cat B \comma \intcat C_0 \); see Corollary
6.2.5 \textit{ibid}. This yields an equivalence of categories
\begin{equation}
  \label{hammer}
  \Cat(T_{\intcat C}, \cat B \comma \intcat C_0) 
    \iso \Cat(T,\cat B) \comma \intcat C, 
\end{equation}
and since pullback-stable (regular) epimorphisms and effective descent remain
unchanged on slice categories (more precisely, \( \cat C/x \to \cat C \)
creates each of the three types of morphism), we can deduce facts about
effective descent of complicated internal multicategories in terms of
simpler ones.

For the remainder of this section, we will discuss some simple examples of
interest, compare our work with other literature, then mention some open
problems.

\vspace{0.5cm}

\paragraph*{$(M\times -)$\bf{-multicategories}}

Given a monoid \(M\), we can define a cartesian monad \( M \times - \) on \(
\Set \). An \( (M\times -) \)-multicategory \( \intcat C \) is, intuitively, a
category with weighted morphisms. \( (M\times-) \)-morphisms are of the form \(
f \colon x \xrightarrow{m} y \) for objects \(x,y\) and an element \(m\in M
\), and if \( g \colon y \xrightarrow{n} z \), then \( g \circ f \colon x
\xrightarrow{n \cdot m} z \). Identities are given by \( \id \colon x
\xrightarrow{1} x \), and these are to satisfy associativity and identity
laws. 

Despite being a more complicated structure than a category, \((M\times
-)\)-functors of effective descent are not harder to come by compared to
ordinary functors. A well-known result (which can be deduced from
(\ref{hammer})) is that \( \Cat(M\times -,\Set) \iso \Cat \comma M \), where
we view \(M\) as a one object category. Hence, an \( (M\times -) \)-functor is
an effective descent morphism whenever it has the respective property as a
functor. In fact, since \cite{Cre99} characterizes effective descent functors,
we have also characterized effective descent \( (M\times-) \)-functors. The
arguments remain unchanged when we replace \( \Set \) by a lextensive category
\( \cat B \) (with regular epi-mono factorizations for the complete
characterization).

\vspace{0.5cm}
\paragraph*{\bf{Ordinary and operadic multicategories}}

A multicategory \( \intcat C \) consists of sets \( \intcat C_0 \) and \(
\intcat C_1 \) of objects and multimorphisms, respectively, together with
domain and codomain functions \( d_1 \colon \intcat C_1 \to \word \intcat C_0
\), \( d_0 \colon \intcat C_1 \to \intcat C_0 \), together with composition
and unit operations \( d_1 \colon \intcat C_2 \to \intcat C_1 \) and \( s_0
\colon \intcat C_0 \to \intcat C_1 \) satisfying associativity and identity
properties. Here, \( \intcat C_2 \) is the set of multicomposable pairs given
by the pullback of \( d_1 \) and \( Td_0 \). Likewise, \( \intcat C_n \) is
the set of multicomposable \(n\)-tuples.

A multicategory functor \( F \colon \intcat C \to \intcat D \) is given by
a pair of functions on objects and multimorphisms which preserve domain,
codomain, unit and composition. Our main result states that \( F \) is
effective for descent whenever it is surjective on multimorphisms,
multicomposable pairs, and multicomposable triples.

To extend this result using (\ref{hammer}), suppose we have an operad \( \cat
O \) (a multicategory with one object). The induced monad \( \word_{\mathcal
O} \) is said to be an \textit{operadic monad}, which is cartesian. These are
related to \textit{strongly regular theories}; we refer the reader to
\cite{Lei04} and \cite{CJ95} for details. One could denote the category \(
\Cat(\word_{\mathcal O}, \Set) \) as the category of operadic multicategories
and functors between them. These functors come with an underlying
multicategory functor, and is effective for descent in \( \Cat(\word_{\mathcal
O},\Set) \) if and only if it is effective for descent in \( \Cat(\word,\Set)
\). As in the previous case, the same arguments work for \( \cat B \)
lextensive.

\vspace{0.5cm}
\paragraph*{\bf{State of the art}}

Our results have shown that three levels of ``surjectivity'' (of singles,
pairs and triples of multimorphisms) are sufficient to determine effective
descent in generalized multicategories. This is consistent with the findings
of \cite[6.2 Proposition]{JST04} for \( \Cat \), and in \cite[Theorem
6.2.9]{Cre99} for \( \Cat(\cat C) \) where \( \cat C \) has finite limits and
a (regular epi, mono)-factorization, where these three levels are also
necessary.

This is also the case for \(\cat V\)-categories, with \(\cat V\) cartesian, as
verified by \cite[Theorem 9.11]{Luc18} (with suitable \( \cat V \)
lextensive), and \cite[Theorem 5.4]{CH04} (with \( \cat V \) a complete
Heyting lattice). In the latter case, since \( \cat V \) is thin, surjectivity
on triples of morphisms is no longer required.

In the enriched multicategory case, for \( T \) the ultrafilter monad and \(
\cat V = 2 \) (so that \((T,\cat V)\dash \Cat = \Top\)), we have the result of
\cite[Theorem 5.2]{CH02}, which requires only two levels of surjectivity as
well.

\vspace{0.5cm}
\paragraph*{\bf{Further work}} 

We also take the opportunity to state some open problems. One might be
interested in verifying whether the converses to Theorems \ref{first.theorem}
or \ref{main.theorem.2} hold. As mentioned in the introduction, LeCreurer gave
an affirmative answer for \(T=\id\) and further requiring a (regular
epi, mono)-factorization on \( \cat B \). One might also wonder if this extra
condition is necessary.

Another interesting problem is to check whether LeCreurer's tools are also
amenable to fully characterize effective descent morphisms of enriched
categories internal to \( \cat B \).

  \appendix\section{$T$-stability of pullback-stable classes}
    \label{Appendix}
    The purpose of this appendix is to establish a couple of auxiliary lemmas
about preservation of pullback-stable classes. Let \( T = (T,e,m) \) be a
cartesian monad on \( \cat B \).

\begin{lemma}
  \label{create.image.stable}
  \(T\) creates any pullback-stable property of morphisms in its essential image.
\end{lemma}
\begin{proof}
  If \( Tf \) satisfies a property \(P\), stable under pullback, then the unit
  and multiplication naturality squares guarantee that \(f\) and \(TTf\) also
  satisfy \(P\).
\end{proof}

\begin{corollary}
  If \(Tf\) is a pullback-stable (regular) epimorphism, effective for descent,
  then \(f\) and \(TTf\) also have the respective property.
\end{corollary}

\begin{lemma}
  \label{lemmaB3}
  Let \( f \colon x \to y \) be a \(T\)-graph morphism, and let \( \cat E \)
  be a class of epimorphisms, containing all retractions, closed under
  composition and cancellation. If \( f_1 \) is in \( \cat E \), then so is
  \(f_0\).
\end{lemma}

\begin{proof}
  Since \( d_0 \colon x_1 \to x_0 \) is a retraction, \( d_0f_1 = f_0d_0 \) is
  in \( \cat E \), therefore so is \(f_0\) by cancellation.
\end{proof}

We are interested in the cases when \( \cat E \) is the class of
pullback-stable epimorphisms, of descent morphisms and of effective descent
morphisms.

  \section*{Acknowledgements}
    \label{Acknowledgements}
    The authors would like to thank Maria Manuel Clementino for her feedback on
this work. We also thank her for fruitful discussions on effective descent
morphisms during our research stay in Oberwolfach. These discussions
positively influenced the revision of this paper, especially concerning the
statement of Corollary \ref{mmc}.

  % Bibliography

\end{document}